\documentclass[leqno,12pt]{article}

\usepackage{amstext    }
\usepackage{amsthm    }
\usepackage{a4}
\usepackage[mathscr]{eucal}
\usepackage{mathrsfs}
\usepackage{url}
\usepackage{amsmath}
\usepackage{amssymb}
\usepackage{amscd}

\numberwithin{equation}{section}

\newtheorem{theorem}{Theorem}[section]
\newtheorem{definition}[theorem]{Definition}
\newtheorem{proposition}[theorem]{Proposition}
\newtheorem{corollary}[theorem]{Corollary}
\newtheorem{lemma}[theorem]{Lemma}

\newcommand{\B}{\mathbb{B}}

\newcommand{\C}{\mathbb{C}}

\newcommand{\N}{\mathbb{N}}

\newcommand{\R}{\mathbb{R}}

\renewcommand{\P}{\mathbb{P}}

\title{Intersection of positive closed currents \\
 of higher bidegree}

\author{Duc-Viet Vu}

\begin{document}

\maketitle

\begin{abstract} Let $X$ be a compact K\"ahler manifold of dimension $n.$ Let $T$ and $S$ be two positive closed currents on $X$ of bidegree $(p,p)$ and $(q,q)$ respectively with $p+q\le n.$ Assume that $T$ has a continuous super-potential. We prove that the wedge product $T \wedge S,$ defined by Dinh and Sibony, is a positive closed current.
\end{abstract}

\noindent
%{\bf Classification AMS 2010}: 35P99, 37A35. 


%\tableofcontents
\section{Introduction}
Let $X$ be a compact K\"ahler manifold of dimension $n.$ %and $\omega$ be a K\"ahler  form on $X$. 
Let $T$ and $S$ be two positive closed currents on $X$ of bidegree $(p,p)$ and $(q,q)$ respectively  with $p+q\le n$. In \cite{Demailly2}, Demailly asked  the question to define the intersection  $T \wedge S.$ The theory of intersections of currents of bidegree $(1,1)$ is well developed, see, e.g., \cite{Bedford_Taylor, Chern_Levine_Nirenberg, Demailly_ag, Fornaess_Sibony}. So the question of Demailly concerns currents of higher degree. 

The problem was recently solved by Dinh and Sibony in \cite{DS_superpotential} using their theory of super-potentials (see also \cite{DSfourier}). Assume that $T$ has continuous super-potentials (see \cite{DS_superpotential} or Section \ref{sec2} for the terminology). Then the wedge product $T \wedge S$ is well-defined. It is known that this product is the difference of two positive closed currents. The operator  satisfies basic properties like the commutativity and the associativity when intersect several currents. The Hodge cohomology class of $T \wedge S$ is the cup product of the ones of $T$ and $S.$ Moreover, $T \wedge S$ depends continuously on $S.$ Therefore, it is positive when $S$ can be approximated by smooth positive closed forms. The last property of approximation is satisfied when $X$ is a homogeneous manifold and also in the case of some dynamical Green currents. The purpose of this work is to prove the positivity of $T \wedge S$ in the general setting. We have the following theorem. 
\begin{theorem} \label{the_positivity} Let $X$ be a compact K\"ahler manifold of dimension $n$. Let $T$ and $S$ be  two positive closed currents on $X$ of bidegree $(p,p)$ and $(q,q)$ respectively with $p+q \le n.$ Assume that $T$ has a continuous super-potential. Then the intersection current $T \wedge S$ is a positive closed current of bidegree $(p+q, p+q)$. 
\end{theorem}
In Section \ref{sec2}, we will recall some basic properties of positive closed currents and their super-potentials. In Section \ref{sec3}, we will introduce an alternative definition of $T \wedge S$ which is a positive closed current. We then show that this definition is equivalent to the one by Dinh and Sibony. The above result will follow immediately. % prove a formula for computing  $T \wedge S$ which gives immediately the positivity of $T \wedge S$  %which give a positive closed current. Then we show that the new definition is equivalent to Dinh-Sibony's one
%(see Proposition \ref{equ_haidinhnghiatuongduong}). Theorem \ref{the_positivity} is then a direct consequence of it. The idea leading to this formula is as follows. 
We will present now the main idea. 

Suppose first that $T$ and $S$ are positive closed smooth forms of $X$. %Let $\Delta$ be the diagonal of $X \times X.$ 
Let $\pi_j$ ($j=1,2$) be the projections from $X \times X$ to the first and second components respectively. We have $T \otimes S= \pi_1^* (T) \wedge \pi_2^* (S).$ This is a positive closed smooth form on $X \times X.$ Then one can compute $T \wedge S$ via the formula 
\begin{align} \label{equ_vietointro}
T \wedge S= (\pi_j)_*(T \otimes S \wedge [\Delta]) \text{ for } \, j=1,2,
\end{align} 
where $[\Delta]$ is the current of integration on the diagonal $\Delta$ of $X \times X.$ 

Observe that because of $[\Delta],$ the formula (\ref{equ_vietointro}) can not  be extended to general singular currents $T$ and $S.$ We can however use the theory of intersection with $(1,1)$-currents if in the place of $\Delta$ we have a hypersurface. This is the reason why we consider the blow-up $\widehat{X \times X}$ of $X \times X$ along $\Delta.$ Let  $\Pi$ be the natural projection from $\widehat{X \times X}$ to $X \times X$ and  $\widehat{\Delta}= \Pi^{-1}(\Delta)$ be the exceptional hypersurface. Recall from \cite{Blanchard, Voisin1} that the blow-up of a compact K\"ahler manifold along a submanifold is also K\"ahler. Let $\widehat{\omega}$ be a K\"ahler form of  $\widehat{X \times X}$.  Observe that $\Pi_{*}(\widehat{\omega}^{n-1} \wedge [\widehat{\Delta}])$ is a non-zero positive closed current of $X \times X$ supported on $\Delta$ and has the same dimension as $\Delta.$ Therefore, it equals a constant times $[\Delta],$ see, e.g., \cite{Demailly_ag}. By normalizing $\widehat{\omega},$ we can suppose that 
\begin{align} \label{equ_omega}
\Pi_{*}(\widehat{\omega}^{n-1} \wedge [\widehat{\Delta}])=[\Delta].
\end{align}
 Put $\widehat{T \otimes S}= \Pi^*(T \otimes S)$ and $\Pi_j= \pi_j \circ \Pi$ ($j=1,2$).  Then (\ref{equ_vietointro}) can be rewritten as 
\begin{align} \label{equ_vietointro2}
T \wedge S= (\Pi_j)_*( \widehat{T \otimes S} \wedge \widehat{\omega}^{n-1} \wedge [\widehat{\Delta}]).
\end{align}

In general, when $T$ and $S$ are only positive closed currents, one still can define  $\widehat{T \otimes S}$ as a positive closed current  outside $\widehat{\Delta}$ and extend it by $0$ through  $\widehat{\Delta}$. We can show that $ \widehat{T \otimes S} \wedge \widehat{\omega}^{n-1} \wedge [\widehat{\Delta}]$ is well-defined provided that $T$ has a continuous super-potential. In this case, we can use (\ref{equ_vietointro2}) as an alternative definition of $T \wedge S$ which gives a positive closed current, see Corollary \ref{co_giaovoiduongcheo}. Proposition  \ref{pro_haidinhnghiatuongduong} below shows that this definition is equivalent to the one of Dinh and Sibony.

\vskip 0.5cm
\noindent
{\bf Acknowledgement.} The author  would like to thank Tien-Cuong Dinh for his valuable help during the preparation of this paper. This research is supported by grants from R\'egion Ile-de-France. 
\section{Super-potential of positive closed currents} \label{sec2}
We will recall now some basic facts and refer to \cite{DS_superpotential} for details. Let $X$ be a compact K\"ahler manifold of dimension $n$ and $\omega$ be a K\"ahler  form on $X$.   It is well-known that the de Rham cohomology of currents and smooth forms  are canonically equal (see \cite[Chap. 3]{GH}). Denote them by $H^r(X,\C)$ with $0 \le r \le n.$ For any closed current $T$ of degree $r,$ denoted by $\{T\}$ its cohomology class in $H^r(X, \C).$ Let $H^{p,p}(X,\R)$ be the vector subspace of $H^{p,p}(X,\C)$ spanned by the classes of closed real $2p$-forms. %Then we have  $H^{p,p}(X,\C)= H^{p,p}(X,\R) \otimes_{\R} \C.$ 
Since a closed positive $(p,p)$-current is real, its class belongs to $H^{p,p}(X,\R).$ If $V$ is an analytic subset of $X$ of dimension $n-p,$ it defines a positive closed current $[V]$ of bidegree $(p,p)$  by integration over $V.$ Its class will be denoted by $\{V\}$ for simplicity.

Let $\mathcal{C}_p$ be the convex cone of positive closed $(p,p)$-currents on $X$ and $\mathcal{D}_p$ be the real vector space generated by $\mathcal{C}_p.$ Since the K\"ahler form $\omega$ is strictly positive, the set $\mathcal{D}_p$ contains all real closed smooth $(p,p)$-forms.  Let $\mathcal{D}^0_p$ be the subspace of $\mathcal{D}_p$ of currents belonging to the class $0$ in $H^{p,p}(X,\R).$ We recall the notion of  $*$-norm on $\mathcal{D}_p.$ Consider first a positive closed current $S$ in $\mathcal{D}_p. $ Define its $*$-norm  by 
$$\|S\|_{*}= |\langle  S, \omega^{n-p} \rangle|$$
which is equal to the mass of $S$. In general, since any $S \in \mathcal{D}_p$ can be written as the difference of two positive closed currents, define
$$\|S\|_{*}= \inf ( \|S^+\|_*+ \|S^-\|_* ),$$
where the infimum is taken over all $S^+, S^- \in \mathcal{C}_p$ such that $S= S^+ - S^-.$ By compactness property of positive closed currents, the above infimum is attained for some $S^+$ and  $S^-.$ %Any such currents of $*$-bounded norm can act on Borel uniformly bounded forms. 
 We say that $S_k$ converges to $S$ in $\mathcal{D}_p$ for the $*$-topology if $S_k$ converges to $S$ weakly as currents and $\|S_k\|_{*}$ is bounded independently of $k.$ The following result is due to Dinh and Sibony, see \cite[Th. 2.4.4]{DS_superpotential} and also \cite[Th. 1.1]{DinhSibony_regularization}. 
\begin{proposition} \label{pro_convergence*}
There is a positive constant $c$ such that for all $S \in \mathcal{D}_p,$ there exist smooth forms  $S_k \in \mathcal{D}_p$ with $k \in \N$ such that $S_n$ converges weakly to $S$ and $\|S_k\|_{*} \le c \|S\|_{*}$ for all $k.$
 \end{proposition}
Let $T$ be in $\mathcal{D}_p$ and $R$ be in $\mathcal{D}^0_q.$ By $dd^c$-lemma for currents (see  \cite[Th. 1.2.1]{Gillet_Soule}), there is a real  $(q-1,q-1)$-current $U_R$ such that $dd^c U_R= R.$ We call $U_R$ a potential of $R.$ Consider the following important example of $R.$ Let   $V$ be a hypersurface of $X$ and $\beta_0$ be a smooth form of the same cohomology class with $[V].$ Then $R= [V]-\beta_0$ is in $\mathcal{D}^0_1.$  One can construct an explicit potential $U_R$ as follows. Consider the holomorphic line bundle of $X$ associated with $V$ and $\sigma$ a holomorphic section whose divisor is $V.$ Take a smooth Hermitian metric on this line bundle and denote by $| \cdot|$ the norm induced by this metric. By Poincar\'e-Lelong formula, there is a smooth form $\beta_1$ such that 
$$dd^c \log |\sigma|= [V]- \beta_1.$$
%(recall that $d= \partial +\overline{\partial}$ and $d^c=\frac{1}{2i \pi}(\partial -\overline{\partial})$). 
Since $\{\beta_0\}=\{V\}= \{\beta_1\},$ there is a smooth function $f$ on $X$ such that $dd^c f= \beta_0 - \beta_1.$ The function $U_R:= \log |\sigma| - f$ is a potential of $R.$   %In suitable local coordinates $x=(x_1,\cdots,x_k)$ of $X,$ we can write 
%\begin{align} \label{euq_divisor}
%\sigma(x)= e^{a(x)}\big(b_0(x)+ b_1(x) x_k+ \cdots +b_m(x) x_k^m\big),
%\end{align}
%for some integer $m,$ where $a(x), b_0(x), \cdots, b_m(x)$ are holomorphic functions and $b_0(0)=\cdots=b_m(0)=0.$  Hence, with the above choice of $S,$ its potential  $U_S$ is an element in $L^1(X)$ and smooth outside $V.$ Moreover, it is clear from (\ref{euq_divisor}) that 
Note that $U_R$ is smooth outside $V$ and if $\sigma'$ is a holomorphic function on an open neighborhood $W$ of a point of $V$ such that its divisor is $V \cap W,$ then  
\begin{align} \label{ine_landautienvoiphi}
U_R(x)- \log |\sigma'|  \text{ is smooth on $W$.} %\le -  c\log  \dist (x, V) %\text{ and } |\nabla U_S(x)| \lesssim  \dist (x, V)^{-1}, 
\end{align}
%for $x$ near $V$ and for some positive constant $c$ independent of $x$. % where $\nabla$ is the gradient of $U_S$ on a fixed atlas.
 
Consider now a current $R \in \mathcal{D}^0_{n-p+1}$ and an $(n-p,n-p)$-current $U_R$ which is a potential of $R$. Let $\alpha=(\alpha_1,\cdots,\alpha_h)$ with $h= \dim H^{p,p}(X,\R)$ be a fixed family of real smooth closed $(p,p)$-forms such that the family of classes $\{ \alpha\}= (\{\alpha_1\}, \cdots, \{\alpha_h\})$ is a basis of $H^{p,p}(X,\R).$ By adding to $U_R$ a suitable closed smooth form, we can assume that $\langle U_R, \alpha_i \rangle=0$ for $i=1,\cdots,h.$ We say that $U_R$ is $\alpha$-normalized. 
\begin{definition}\label{de_superpotential} (\cite[Def. 3.2.2]{DS_superpotential}) Let $T$ be a current in $\mathcal{D}_p$ as above. The $\alpha$-normalized super-potential $\mathcal{U}_T$ of $T$ is the function defined on smooth forms $R \in \mathcal{D}^0_{n-p+1}$ and given by 
$$\mathcal{U}_T(R)= \langle T, U_R \rangle,$$
where $U_R$ is an $\alpha$-normalized smooth potential of $R.$ We say that $T$ has a continuous super-potential if $\mathcal{U}_T$ can be extended to a function on $\mathcal{D}^0_{n-p+1}$ which is continuous with respect to the $*$-topology. In this case, the extension is also denoted by $\mathcal{U}_T.$
\end{definition}  
 By  \cite[Lem. 3.2.1]{DS_superpotential}, $\mathcal{U}_T(R)$ does not depend on the choice of an $\alpha$-normalized $U_R.$ And the continuity of $\mathcal{U}_T$ does not depend on $\alpha.$  Observe that when $\{T\}=0,$  the $\alpha$-normalized super-potential of $T$ does not depend on $\alpha.$ Indeed, in this case, it is the restriction of any potential $U_T$ of $T$ to the set of smooth forms in  $\mathcal{D}^0_{n-p+1}.$ Assume that $T$ has a continuous super-potential. Take any current $S \in \mathcal{D}_{q}.$ Let $(a_1,\cdots, a_h)$ be the coefficients of $\{T\}$ in the basis $\{\alpha\}.$  Define $T \wedge S$ to be the real $(p+q,p+q)$-current satisfying
\begin{align} \label{defi_wedgecurrents}
\langle T \wedge S, \Phi \rangle := \mathcal{U}_T\big( dd^c \Phi  \wedge S \big)+ \sum_{1 \le j \le h} a_j \langle \alpha_j, \Phi \wedge S \rangle,
\end{align}
for any real smooth $(n-p-q, n-p-q)$-form $\Phi.$ %This definition coincides with the usual wedge product as at least $T$ or $S$ is smooth and it satisfies the associative property (Proposition 3.3.4 \cite{DS_superpotential}). %However, it is not clear from the above definition whether the wedge product of two positive closed currents  one of which has a continuous super-potential is positive. In this paper, we will propose another construction for the intersection of two currents (one of which has a continuous super-potential) which is equivalent to the aforementioned one. The nature of this new construction will give us an affirmative answer for the above question.
%This theorem is a direct corollary of a construction of the intersection of currents in the next section (see Proposition \ref{equ_haidinhnghiatuongduong}), which can be seen as an another equivalent way to define the wedge product of  $T$ and $S.$

\section{Alternative definition for the intersection of currents} \label{sec3}
Let $X, \widehat{X \times X}, \omega, \widehat{\omega}, \Pi, \Pi_j, \pi_j, \Delta, \widehat{\Delta}$ be as in the previous sections. Consider two currents $T \in \mathcal{D}_p$ and $S \in \mathcal{D}_q$ as above with $p+q \le n.$ Let $h, a_j$ and $\alpha_j$ with $1 \le j \le h$ be as in the last section. From now on, assume that $T$ is positive and has a continuous super-potential.
%Let $T$ and $S$ be in $\mathcal{D}_p$ and $\mathcal{D}_{q}$ respectively with $p+q \le n$. Assume that $T$ has a continuous super-potential. Let $\Delta$ be the diagonal of $X \times X.$ %It is a submanifold of dimension $n$ of $X \times X.$
 %Let $\Pi: \widehat{X \times X} \rightarrow X \times X$ be the blow-up of $X \times X$ along $\Delta.$ Let $\pi_j$ ($j=1,2$) be the projections of $X \times X$ on the first and second factors respectively. 
Note that $\Pi_j=\pi_j \circ \Pi$ are submersions, for a proof see \cite{DS_superpotential} or the proof of Lemma \ref{le_ukhatichvoiS} below. %Indeed, this can be seen by using special local coordinates $(z,y)$ on $\widehat{X \times X}$ such that $\Pi_2(z,y)=y,$ see the proof of Lemma \ref{le_ukhatichvoiS} or \cite{DS_superpotential} for details. Hence, $\Pi_2$ is a submersion. By symmetry, $\Pi_1$ is also a submersion. 
Define $\widehat{T}= \Pi_1^{*}(T)$ and $\widehat{S}= \Pi_2^*(S).$ They are positive closed currents on $\widehat{X \times X}.$ Put $\widehat{\alpha}_j=\Pi_1^*(\alpha_j)$ for $1 \le j \le h.$ %Since $\Pi$ is biholomorphic outside $\widehat{\Delta},$ we see that $\{\widehat{\alpha}_j\}$ are linearly independent in $H^{p,p}(\widehat{X \times X}, \R).$ 
\begin{lemma} \label{le_laiphaisuaconU_T} The current $\widehat{T}$ has a continuous super-potential. %The classes $\{\widehat{\alpha}_j\}$ are linearly independent in $H^{p,p}(\widehat{X \times X}, \R).$  Complete them to be a basis $\widehat{\alpha}'$ of $H^{p,p}(\widehat{X \times X}, \R).$ Let $\mathcal{U}_{\widehat{T}}$ be the $\widehat{\alpha}'$-normalized  super-potential of $\widehat{T}.$ Then $\mathcal{U}_{\widehat{T}}$ is continuous. 
\end{lemma}
\begin{proof} Suppose that the classes $\{\widehat{\alpha}_j\}$ are linearly dependent. Then there exist real numbers $b_j$ with $1 \le j \le h$ which are not simultaneously equal to zero and a smooth form $\widehat{\gamma}$ such that $\sum_{j=1}^h b_j \widehat{\alpha}_j= d(\widehat{\gamma}).$ Taking the wedge product with $\widehat{\omega}^n$ in the last equality and then using the push-forward by $(\Pi_1)_*$  give
\begin{align}\label{equ_ngaiquatroilun}
\sum_{j=1}^h b_j  \alpha_j \wedge  (\Pi_1)_*(\widehat{\omega}^n) = d\big( (\Pi_1)_*( \widehat{\gamma} \wedge \widehat{\omega}^n) \big).
\end{align}
Note that $(\Pi_1)_* \widehat{\omega}^n$ is actually a nonzero constant since $\widehat{\omega}^n$ is closed and positive. We deduce that the left-hand side of (\ref{equ_ngaiquatroilun}) is a non-trivial linear combination of $\alpha_j,$ $1 \le j \le h.$ However this contradicts the fact that $\{\alpha_j\}$ are linearly independent. Hence, the classes $\{\widehat{\alpha}_j\}$ are linearly independent. Complete them to be basis  $\widehat{\alpha}'$ of $H^{p,p}(\widehat{X \times X}, \R).$ Let $\mathcal{U}_{\widehat{T}}$ be the $\widehat{\alpha}'$-normalized super-potential of $\widehat{T}.$  

Put $\alpha_T= \sum_{j=1}^h a_j \alpha_j$ and $\widehat{\alpha}_T= \Pi_1^* \alpha_T.$ Remark that $\alpha_T$ and $\widehat{\alpha}_T$ are in the same cohomology classes with $T$ and $\widehat{T}$ respectively. Let $U_{T-\alpha_T}$ be a potential of  $T-\alpha_T.$ Then $U_{\widehat{T}- \widehat{\alpha}_T}:=\Pi_1^* U_{T-\alpha_T}$ is a potential of $\widehat{T}- \widehat{\alpha}_T.$  By definition, for any smooth form $\tilde{R} \in  \mathcal{D}^0_{2n-p+1}(\widehat{X \times X}),$ we have 
$$\mathcal{U}_{\widehat{T}}(\tilde{R})= \langle \widehat{T}, U_{\tilde{R}}  \rangle=\langle \widehat{T}- \widehat{\alpha}_T, U_{\tilde{R}} \rangle= \langle U_{\widehat{T}- \widehat{\alpha}_T}, \tilde{R} \rangle$$
By our choice of potentials, the last quantity equals 
$$\langle U_{T-\alpha_T}, (\Pi_1)_*\tilde{R} \rangle= \mathcal{U}_T((\Pi_1)_*\tilde{R}).$$
The continuity of $\mathcal{U}_T$ now implies immediately the same property for $\mathcal{U}_{\widehat{T}}.$ The proof is finished. 
\end{proof}
Thanks to Lemma \ref{le_laiphaisuaconU_T}, one can define $\widehat{T} \wedge \widehat{S}$ as in (\ref{defi_wedgecurrents}).  Recall that $T \otimes S$ is a positive closed $(p+q,p+q)$-current on $X \times X$ depending continuously on $T$ and $S.$ Its action on smooth forms can be described as follows. 
Let $x$ be local coordinates of $X.$ They induce naturally local coordinates $(x,y)$ on $X \times X.$ For a smooth form $\Phi(x,y)$ of $X \times X,$ we have 
\begin{align} \label{equ_dinhnghiatenxo}
\langle T \otimes S, \Phi \rangle=\big \langle T, S\big(\Phi(x,\cdot)\big) \big \rangle=\big \langle S, T\big(\Phi(\cdot,y)\big) \big \rangle.
\end{align} 
%By positivity of $T$ and $S,$ the above formula also holds for Borel differential forms whose coefficients are bounded uniformly. 
Let $\Pi'$ be the restriction of $\Pi$ to 
$\widehat{X \times X} \backslash \widehat{\Delta}.$ The current
$$\widehat{T \otimes S}= \Pi'^{*} (T \otimes S)$$
 is well-defined and positive closed  on $\widehat{X \times X} \backslash \widehat{\Delta}$ because $\Pi'$ is biholomorphic. By Proposition 5.1 of \cite{DinhSibony_pullback},  the mass of $\widehat{T \otimes S}$ is bounded. Hence, it  can be extended by zero to be a positive closed current of $\widehat{X \otimes X}$ through $\widehat{\Delta},$ see \cite{Demailly_ag, Sibony, Skoda}. We still denote by $\widehat{T \otimes S}$ the extended current.   Take a smooth closed  $(1,1)$-form $\widehat{\beta}$ with $ \{\widehat{\beta}\}=\{ \widehat{\Delta}\}.$ Since $\widehat{\Delta}$ is a hypersurface, choose  a potential $\hat{u}= U_{[\widehat{\Delta}] - \widehat{\beta}}$ of  $[\widehat{\Delta}] - \widehat{\beta}$ as in Section \ref{sec2}. It is smooth outside $\widehat{\Delta}$ and its behaviour  near $\widehat{\Delta}$ is described by (\ref{ine_landautienvoiphi}). By adding a constant to $\hat{u}$ if necessary, we can assume that $\hat{u} \le -1.$ 
\begin{lemma} \label{le_ukhatichvoiS} The current $\hat{u} \widehat{S}$ is well-defined. Moreover, if smooth forms $S_k \in \mathcal{D}_q$ converge to $S$ in the $*$-topology,  then $\hat{u} \widehat{S}_k$ converge  weakly to $\hat{u} \widehat{S}.$ %And if  continuous functions $\hat{v}_k$ decrease pointwise to $\hat{u},$ then $\hat{v}_k \widehat{S}$ converges weakly to $\hat{u} \widehat{S}.$
 %Let $D \hat{u}$ be  a partial derivative of $\hat{u}$ of order $1.$  We have the same conclusion for the currents $ \max\{D \hat{u},0\} \widehat{S}$ and $\max\{-D \hat{u},0\} \widehat{S}.$ 
\end{lemma}
\begin{proof} %The second statement follows easily from the proof for the first one. Hence 
We prove the first assertion.   For any smooth $(2n-q,2n-q)$-form $\hat{\eta}$ on $\widehat{X \times X},$ we will show that $(\Pi_2)_{*}( \hat{u} \hat{\eta})$ is  a smooth form on $X.$ This allows us to  define  
\begin{align} \label{equ_le31}
\langle \hat{u} \widehat{S}, \hat{\eta}   \rangle= \langle S, (\Pi_2)_*( \hat{u} \hat{\eta}) \rangle.
\end{align}
To see that $(\Pi_2)_{*}( \hat{u} \hat{\eta})$ is smooth, we just need to work locally. Consider local coordinates $(W,x=(x_1,\cdots, x_n))$ on a chart $W$ of $X.$ Without loss of generality, we can suppose $W$ is diffeomorphic  to the unit ball $\B_1$ in $\C^n.$ Consider induced local coordinates $(x,y)$ on $W \times W.$ We have  $\Delta \cap (W \times W)= \{ x=y\}.$ Define new local coordinates $(x',y)$ on $W \times W$ by putting $x':= x-y.$ Hence $\Delta$ is given by the equation $x'=0.$ The set $\Pi^{-1}(W \times W)$ is biholomorphic to the manifold $M$ in $\C^n \times \C^n \times \P^{n-1}$ defined by 
$$M= \big \{(x', y, [v]):  y \in \B_1,\, x'+y \in \B_1,\,  [v] \in  \P^{n-1} \text{ and } x' \in [v] \big \},$$
where $[v]=[v_1 : v_2: \cdots :v_{n}]$ denotes the homogeneous coordinates of $\P^{n-1}.$  %Write $[v]=[1: v_2: \cdots : v_n].$
 Let $M_j$ ($1 \le j \le n$) be the open subset of $M$ containing all points $(x',y,[v]) \in M$ with $v_j \not =0.$  They form an open covering of $M.$ For $(x',y,[v]) \in M_1,$ we have $x'_1 v_j = x'_j v_1.$ Choose $v_1=1,$ then $x'_j= x'_1 v_j.$ We deduce that  $\big(x'_1, v_2, \cdots, v_{n}, y\big)$ are coordinates on $M_1$ and  $\widehat{\Delta} \cap  M_1=\{x'_1=0\}.$  %$\Pi(x_1', v_2, \cdots,v_n, y)=(x_1', x'_1 v_2, \cdots, x'_1 v_n,y)$ and  
Since $\Pi_2(x'_1, v_2,\cdots, v_n,y)= y,$ we see  that
\begin{align*}
(\Pi_2)_{*}( \hat{u} \hat{\eta}) &= \int_{x_1',v_2,\cdots, v_n} \hat{u}(x'_1, v_2,\cdots, v_n,y)\hat{\eta}(x'_1, v_2,\cdots, v_n,y)\\
&= \int_{x_1',v_2,\cdots, v_n} \log|x'_1| \hat{\eta}(x'_1, v_2,\cdots, v_n,y) \\
&+\quad \int_{x_1',v_2,\cdots, v_n} \hat{u}'(x'_1, v_2,\cdots, v_n,y) \hat{\eta}(x'_1, v_2,\cdots, v_n,y),
\end{align*}
where $ \hat{u}'(x'_1, v_2,\cdots, v_n,y)$ is a smooth function,  see (\ref{ine_landautienvoiphi}). This implies that the last integral defines a smooth form in $y.$ It is also clear that the integral involving $\log|x'_1|$ depends smoothly in $y$. The proof of the first assertion is finished. The second assertion is a direct consequence of the identity (\ref{equ_le31}).
 %Write 
%$$\hat{\eta}=\sum_{I}\hat{\eta}_I d(x_1',v) \wedge d\overline{(x_1',v)} \wedge d y_I  \wedge d \overline{y}_I + \hat{\eta}',$$
 %where $d(x_1',v)$ is the short form for $d x'_1 \wedge d v_2 \wedge  \cdots \wedge d v_n,$ $\hat{\eta}'$ does not contains forms of this type
 %and the sum is taken over all $I \subset \{1,\cdots, n\}$ of cardinality $n-q$. We have
%\begin{multline*}
 %\int_{x_1',v_2,\cdots, v_n} \log|x'_1| \hat{\eta}(x'_1, v_2,\cdots, v_n,y)\\
%= \sum_{I}\Big( \int_{x_1',v_2,\cdots, v_n} \log|x'_1| \hat{\eta}_I(x'_1,v,y) d(x_1',v) \wedge d\overline{(x_1',v)} \Big) d y_I  \wedge d \overline{y}_I
%\end{multline*}
%which is a smooth form of $y.$ Hence, $\hat{u} \widehat{S}$ is well-defined. %and $\|\hat{u} \widehat{S}\| \le c\|S\|.$ %In the same local coordinates, we have only new difficulty with the partial derivative of $\hat{u}$ with respect to the real part or imaginary part of $x_1'.$ In this case, $D \hat{u} \sim |x_1'|^{-1}.$ Write $D \hat{u}= \max\{ D \hat{u},0 \}- \max\{-D \hat{u}, 0\}$ which is the difference of two positive functions having the same singularity as $D \hat{u}.$ By noting that $$\int_{\{|x_1'| \le 1\}}\frac{d x'_1 d\bar{x'_1}}{|x'_1|} < +\infty,$$ we get the conclusion for $(D \hat{u}) \widehat{S}.$ 
 The proof is finished.
\end{proof}
\begin{proposition} \label{pro_currentbangnhau} We have $\widehat{T} \wedge \widehat{S}= \widehat{T \otimes S}.$
\end{proposition}
\begin{proof} %Firstly, we show that the two currents are equal on $\widehat{X \times X} \backslash \widehat{\Delta}.$ Take a smooth form $\widehat{\eta}$ compactly supported on  $\widehat{X \times X} \backslash \widehat{\Delta}.$ Then we have 
%$$\langle  \widehat{T \otimes S}, \widehat{\eta} \rangle= \big \langle T \otimes S, (\Pi')_{*} \widehat{\eta} \big \rangle= \big \langle T, S \big((\Pi')_{*} \widehat{\eta}(\cdot,y) \big) \big\rangle.$$On the other hand, 
Consider first the case where $S$ is smooth. So  $\widehat{T} \wedge \widehat{S}$ is the usual wedge product of a current with a smooth form. We then see that 
$\widehat{T} \wedge \widehat{S}= \Pi^*( T \otimes S)= \widehat{T \otimes S}$ outside $\widehat{\Delta}.$ Observe that the fibers of the submersion $\Pi_1$ are transverse to $\widehat{\Delta}.$ Therefore, $\widehat{T}$ has no mass on $\widehat{\Delta}.$ Hence, $\widehat{T} \wedge \widehat{S}$ has no mass on $\widehat{\Delta}.$ We deduce that $\widehat{T} \wedge \widehat{S}= \widehat{T \otimes S}$ in this case because  $ \widehat{T \otimes S}$ has no mass on  $\widehat{\Delta}$ by definition. 

In general,  by Proposition \ref{pro_convergence*}, there is a sequence of smooth forms $S_k \in \mathcal{D}_q$ converging to $S$ in the $*$-topology. The first case and the continuity on $S$ imply that  $\widehat{T} \wedge \widehat{S}= \widehat{T \otimes S}$ outside $\widehat{\Delta}.$ It remains to show that the restriction $\mathbf{1}_{\widehat{\Delta}}(\widehat{T} \wedge \widehat{S})$ of $\widehat{T} \wedge \widehat{S}$ vanishes. This is equivalent to say that 
\begin{align} \label{ine_thuxemnao}
\int_{\widehat{\Delta}}\widehat{T} \wedge \widehat{S} \wedge \widehat{\Phi}=0,
\end{align}
for any  smooth form $\widehat{\Phi}$ of bidegree $2n-p-q$. By Proposition \ref{pro_convergence*}, we can write $S= S^+ - S^-$ where $S^+$ and $S^-$ are approximable by smooth positive closed forms. Since $\widehat{T} \wedge \widehat{S}=\widehat{T} \wedge \widehat{S}^+-\widehat{T} \wedge \widehat{S}^-,$ we only need to verify that  $\mathbf{1}_{\widehat{\Delta}}(\widehat{T} \wedge \widehat{S}^{\pm})=0.$ Therefore, without loss of generality, assume that $\widehat{T} \wedge \widehat{S}$ is positive. Consequently, it suffices to prove (\ref{ine_thuxemnao}) for $\widehat{\Phi}= \widehat{\omega}^{2n-p-q}.$

Let $\chi$ be a convex increasing smooth function on $\R$ such that $\chi(t)=0$ if $t \le -1/4,$ $\chi(t)=t$ for $t\ge 1/4$ and $0 \le \chi' \le 1.$  For each positive integer $k,$ put 
$$\hat{u}_k= \chi(\hat{u}+k) - k.$$
This is a smooth  negative quasi-p.s.h. function since $\widehat{u} \le -1$. The functions $\hat{u}_k$ decrease to $\hat{u}$ and $-\hat{u}_k/k $ decrease to  the characteristic function $\mathbf{1}_{\widehat{\Delta}}$ of $\widehat{\Delta}$ as  $k \rightarrow \infty.$ The first property implies that $\widehat{S} \wedge dd^c \hat{u}_k$ converges weakly to $\widehat{S} \wedge dd^c \hat{u},$ see Lemma \ref{le_ukhatichvoiS}. We also have 
 $$dd^c \hat{u}_k= [\chi'(\hat{u}+k)]^2 d \hat{u} \wedge d^c \hat{u}+ \chi''(\hat{u}+k) dd^c \hat{u} \ge \chi''(\hat{u}+k) dd^c \hat{u}\ge -c \widehat{\omega},$$
 for some positive constant $c.$  This yields that $dd^c \hat{u}_k=(dd^c \hat{u}_k+ c \widehat{\omega}) - c\widehat{\omega}$ which is the difference of two positive closed currents in the same cohomology class $c \{\widehat{\omega}\}$. 
We deduce that $dd^c \hat{u}_k$ is $*$-bounded uniformly in $k$ and then so is $\widehat{S} \wedge dd^c \hat{u}_k \wedge  \widehat{\omega}^{2n-p-q}$ because we have
\begin{align} \label{ine_estimateon*norm}
\|\widehat{S} \wedge dd^c \hat{u}_k \wedge  \widehat{\omega}^{2n-p-q}\|_{*} \le c \|S\|_{*} \|dd^c \hat{u}_k\|_*, 
\end{align}
for a positive constant $c$ depending only on $(X, \omega).$  It follows that 
$$\widehat{S} \wedge dd^c \hat{u}_k \wedge  \widehat{\omega}^{2n-p-q} \rightarrow \widehat{S} \wedge dd^c \hat{u} \wedge  \widehat{\omega}^{2n-p-q}$$
in the $*$-topology. The equality (\ref{ine_thuxemnao}) with $\widehat{\Phi}= \widehat{\omega}^{2n-p-q}$  is equivalent to
\begin{align} \label{ine_thuxemnao2}
\big \langle \widehat{T} \wedge \widehat{S}, -\frac{\hat{u}_k}{k} \cdot \widehat{\omega}^{2n-p-q} \big \rangle \rightarrow 0 \text{ as  } k   \rightarrow \infty.
\end{align}
Applying the formula (\ref{defi_wedgecurrents}) to $\widehat{T} \wedge \widehat{S}$ gives 
\begin{align*}
\big \langle \widehat{T} \wedge \widehat{S}, -\frac{\hat{u}_k}{k} \cdot \widehat{\omega}^{2n-p-q} \big \rangle = -\frac{1}{k}\mathcal{U}_{\widehat{T}}\big(\widehat{S} \wedge dd^c\hat{u}_k \wedge  \widehat{\omega}^{2n-p-q}\big)- \frac{1}{k} \big \langle \widehat{\alpha}_T,  \hat{u}_k\widehat{S} \wedge \widehat{\omega}^{2n-p-q}  \big \rangle,
%&\mathcal{U}_{\widehat{T}}\big(\widehat{S} \wedge dd^c\frac{(\hat{u}_n)^{\epsilon}}{n} \wedge  \widehat{\omega}^{2n-p-q}\big)
\end{align*}
where $\widehat{\alpha}_T= \sum_{j=1}^h a_j \widehat{\alpha}_j$. The last quantity converges to $0$ as $k \rightarrow \infty$ for the mass norm of $\hat{u}_k\widehat{S}$ is bounded independently of $k$ by Lemma \ref{le_ukhatichvoiS}. On the other hand, the continuity of $\mathcal{U}_{\widehat{T}}$ gives  
$$\mathcal{U}_{\widehat{T}}\big(\widehat{S} \wedge dd^c \hat{u}_k \wedge  \widehat{\omega}^{2n-p-q}\big) \rightarrow \mathcal{U}_{\widehat{T}}\big(\widehat{S} \wedge dd^c \hat{u} \wedge  \widehat{\omega}^{2n-p-q}\big)$$
which is finite, as $k \rightarrow \infty$. Hence we get  (\ref{ine_thuxemnao2}). The proof is finished.
\end{proof}
\begin{lemma} \label{co_ukhatichdoivoiTS} 
The current $\hat{u}(\widehat{T} \wedge \widehat{S})$ is well-defined. Denote it by $\hat{u}\widehat{T} \wedge \widehat{S}$ for simplicity. %Let $\hat{v}_k$ be  negative continuous functions decreasing pointwise to $\hat{u}$. Then 
%we have  $$\hat{v}_k\widehat{T} \wedge \widehat{S} \xrightarrow[]{k \rightarrow \infty} \, \hat{u}\widehat{T} \wedge \widehat{S}.$$
For any closed real smooth form $\widehat{\Phi}$ of $\widehat{X \times X}$ of the right bidegree, we have %write $dd^c \widehat{\Phi}= \Omega^+- \Omega^-,$ the difference of two positive closed currents. Then we have 
\begin{align}\label{eq_congthuchoPhi}
\langle \hat{u}\widehat{T} \wedge \widehat{S}, \widehat{\Phi}\rangle = \mathcal{U}_{\widehat{T}}\big( dd^c( \hat{u}\widehat{S}  \wedge \widehat{\Phi})  \big)+
\sum_{j=1}^h a_j \langle \widehat{S},  \hat{u} \widehat{\alpha}_j \wedge \widehat{\Phi}\rangle.
\end{align}
In particular, $\langle \hat{u}\widehat{T} \wedge \widehat{S}, \widehat{\Phi}\rangle$ depends continuously on $S.$
\end{lemma}
\begin{proof} Using the computation in the proof of Proposition \ref{pro_currentbangnhau}, we have %We proved there that  
$$\langle \widehat{T} \wedge \widehat{S}, \hat{u} \cdot \widehat{\omega}^{2n-p-q} \rangle=\lim_{k \rightarrow \infty}\mathcal{U}_{\widehat{T}}\big(\widehat{S} \wedge dd^c \hat{u}_k \wedge  \widehat{\omega}^{2n-p-q}\big)+\big \langle \widehat{\alpha}_T,  \hat{u}_k\widehat{S} \wedge \widehat{\omega}^{2n-p-q}  \big \rangle,$$
where $\hat{u}_k$ is defined as in Proposition \ref{pro_currentbangnhau}. The same arguments at the end of the above proposition show that the last limit is finite. The first assertion follows.
 %In other words, $\hat{u}$ is integrable with respect to the trace measure of  $\widehat{T} \wedge \widehat{S}.$  %For each smooth form $\widehat{\Phi}$ of $\widehat{X \times X}$ of the right bidegree, the measure $\widehat{T} \wedge \widehat{S} \wedge \widehat{\Phi}$ is absolutely continuous with respect to the trace measure. 
%The first statement is followed. The proof for (\ref{eq_congthuchoPhi}) is based on the previous proof. It is enough to prove it for $T$ with $\{T\}=0.$ Reasoning as in the previous proof, we have 
Note that each smooth closed form $\Phi$ can be written as the difference of two positive closed forms. Hence it is enough to prove (\ref{eq_congthuchoPhi}) for positive closed forms $\Phi.$ The computations in Proposition   \ref{pro_currentbangnhau} still hold for $\Phi$ in place of $\widehat{\omega}^{2n -p-q}.$ Hence (\ref{eq_congthuchoPhi}) follows.

In order to prove the last assertion, it is enough to prove it for positive closed forms $\Phi$ by the same reason as above. Let $\{S_l\}_{l\in \N}$ be a sequence of currents in $\mathcal{D}_q$ which converges to $S$ in the $*$-topology. Put $\widehat{S}_l= \Pi_2^*(S_l).$  It is clear that $\widehat{S}_l$ converges to $\widehat{S}$ in the $*$-topology. Lemma \ref{le_ukhatichvoiS} implies that $dd^c( \hat{u}\widehat{S}_l  \wedge \widehat{\Phi})$ converges weakly to $dd^c( \hat{u}\widehat{S}  \wedge \widehat{\Phi})$ and %it was showed in in the proof of  Proposition   \ref{pro_currentbangnhau} that 
\begin{align} \label{eq_limittheokl}
\lim_{k \rightarrow \infty}dd^c( \hat{u}_k \widehat{S}_l  \wedge \widehat{\Phi})= dd^c( \hat{u} \widehat{S}_l  \wedge \widehat{\Phi}),
\end{align}
for any $l \in \N.$ Applying (\ref{ine_estimateon*norm}) to $S_k$ in place of $S$, we see that the mass of $dd^c( \hat{u}_k \widehat{S}_l  \wedge \widehat{\Phi})$ is bounded independently of $k$ and $l.$ This combined with (\ref{eq_limittheokl}) yields that the $*$-norm of $dd^c( \hat{u} \widehat{S}_l  \wedge \widehat{\Phi})$ is bounded independently of $l.$ We deduces that  $dd^c( \hat{u}\widehat{S}_l  \wedge \widehat{\Phi})$ converges to $dd^c( \hat{u}\widehat{S}  \wedge \widehat{\Phi})$ in the 
$*$-topology.  The continuity of $\mathcal{U}_{\widehat{T}}$ now implies that  the right-hand side of  (\ref{eq_congthuchoPhi}) depends continuously on $S.$ The proof is finished.

%by similar arguments as in the proof of  Proposition   \ref{pro_currentbangnhau}, one can show that  the $*$-norm of $dd^c( \widehat{u}\widehat{S}  \wedge \widehat{\Phi})$ is bounded independent of 
% \begin{align} \label{equ_xapchiTSvaPhi}
 %\langle \hat{u} \widehat{T} \wedge \widehat{S}, \widehat{\Phi}\rangle=\lim_{k \rightarrow \infty} \mathcal{U}_{\widehat{T}}\big(dd^c( \hat{u}_k \widehat{S}  \wedge \widehat{\Phi})  \big).
 %\end{align}
 %Write $\widehat{\Phi}= \Omega^+- \Omega^-,$ the difference of two positive closed forms. Then we have  
%On the other hand, we have $dd^c(\hat{u}_k\widehat{S} \wedge \widehat{\Phi}) \rightarrow dd^c(\hat{u}\widehat{S} \wedge \widehat{\Phi})$ in the $*$-topology
%as $k \rightarrow \infty$ . By combining with the continuity of $\mathcal{U}_{\widehat{T}},$ the limit in (\ref{equ_xapchiTSvaPhi}) equals $\mathcal{U}_{\widehat{T}}\big(dd^c(\hat{u}\widehat{S} \wedge \widehat{\Phi})  \big).$ The proof is finished.
\end{proof}
 \begin{corollary} \label{co_giaovoiduongcheo}  Define the intersection $\widehat{T \otimes S} \wedge [\widehat{\Delta}]$ by putting
\begin{align} \label{defi_(1,1)current}
 \widehat{T \otimes S} \wedge [\widehat{\Delta}] = dd^c\big( \hat{u}\widehat{T \otimes S} \big)+ \widehat{T \otimes S} \wedge \beta.
\end{align}
Then $\widehat{T \otimes S} \wedge [\widehat{\Delta}]$ is positive when $S$ is positive.
\end{corollary}
\begin{proof} We only need to prove the positivity. This property is classical since the current $[\widehat{\Delta}]$ is of bidegree $(1,1).$ We give here a proof for the sake of the reader.  Fix a small open subset $\widehat{W}$ of $\widehat{X \times X}$ biholomorphic to a ball. We can find a smooth function $\hat{v}$ on $\widehat{W}$ such that $dd^c \hat{v}= \widehat{\beta}.$ Hence the function $\hat{u}'= \hat{u}+ \hat{v}$ satisfies $dd^c \hat{u}'= [\Delta] \ge 0.$ So $\hat{u}'$ is p.s.h. on $\widehat{W}.$ We then have $\widehat{T \otimes S} \wedge [\widehat{\Delta}]= dd^c\big(\hat{u}' \widehat{T \otimes S} \big)$ on $\widehat{W}.$ If $\hat{u}'_k$ is a sequence of smooth p.s.h. functions on $\widehat{W}$ decreasing to $\hat{u}'$, then the last current is the limit of $dd^c\big(\hat{u}'_k \widehat{T \otimes S} \big)$ which is clearly positive since it equals $dd^c\hat{u}'_k \wedge \widehat{T \otimes S}.$ The proof is finished.
\end{proof}
\begin{lemma}\label{le_continuitedeS} Let $Y$ be a closed subset of $X.$ Let $R$ be a positive $(p,p)$-current of $X$ and let $R_k$ be a sequence of positive $(p,p)$-currents of $X$ converging weakly to $R$ as currents in $X \backslash Y.$  Assume that $R$ has no mass on $Y$ and the masses of $R_k$ converge to the one of $R.$ Then $R_k$ converges weakly to $R$ in $X.$
%Let $S_k$ be a sequence of currents  converging to $S$ in $\mathcal{D}_q.$ Then $$\hat{u}\widehat{T \otimes S} \rightarrow \hat{u}\widehat{T \otimes S}$$
\end{lemma}
\begin{proof} For each $\epsilon>0,$ let $Y_{\epsilon}$ be the set of points in $X$ of distance less than $\epsilon$ to $Y.$ Let $\chi_{\epsilon}$ be a continuous function on $X$ such that $0 \le \chi_{\epsilon}\le 1$ and $\chi_{\epsilon}=1$ on $X \backslash Y_{2\epsilon}$ and $\chi_{\epsilon}=0$ on  $\overline{Y}_{\epsilon}.$ Take any continuous real form $\Phi$ on $X$ of  bidegree $n-p$. We need to prove that 
\begin{align} \label{limit_doiqua}
R_k(\Phi) \rightarrow R(\Phi) \quad  \text{ as } \quad k \rightarrow \infty.
\end{align}
Since a continuous form can be written as the difference of two continuous positive forms, we can assume that $\Phi$ is positive. The hypothesis on $R_k$ implies that $R_k( \chi_{\epsilon}\Phi)$ converges to $R(\chi_{\epsilon} \Phi).$ Hence in order to prove (\ref{limit_doiqua}), it is sufficient to show that 
\begin{align}\label{limit2}
\lim_{\epsilon \rightarrow 0} \delta_{\epsilon}=0,
\end{align}
where 
$$\delta_{\epsilon}=\limsup_{k \rightarrow \infty} \int_{\overline{Y}_{2 \epsilon}}R_k(\Phi) \rightarrow 0.$$
Let $\mu_k= R_k \wedge \omega^{n-p}$ and $\mu= R \wedge \omega^{n-p}$ be the trace measures of $R_k$ and $R$ respectively.  Observe that $\delta_{\epsilon}$ is less than a constant times 
$$\limsup_{k \rightarrow \infty}\mu_k( \overline{Y}_{2 \epsilon})=\|R\|- \liminf_{k \rightarrow \infty} \mu_k(X \backslash \overline{Y}_{2\epsilon}).$$  
Since the set $X \backslash \overline{Y}_{2 \epsilon}$ is an open subset of $X \backslash Y,$ the last limit is greater than $\mu(X \backslash \overline{Y}_{2\epsilon}).$ Hence we get 
$$\limsup_{k \rightarrow \infty} \int_{\overline{Y}_{2 \epsilon}}R_k(\Phi)  \lesssim  \|R\|- \mu(X \backslash \overline{Y}_{2\epsilon})= \mu(\overline{Y}_{2 \epsilon}).$$
The last quantity converges to zero as $\epsilon \rightarrow 0$ because $\mu$ has no mass on $Y.$ The proof is finished.
\end{proof}
%Now we claim that 
\begin{proposition} \label{pro_haidinhnghiatuongduong} For $j=1$ or $2,$ we have 
\begin{align} \label{equ_haidinhnghiatuongduong} 
T \wedge S= (\Pi_j)_{*}\big(\widehat{T \otimes S} \wedge [\widehat{\Delta}] \wedge \widehat{\omega}^{n-1}\big),
\end{align}
where $T \wedge S$ is defined as in (\ref{defi_wedgecurrents}).
\end{proposition}
\begin{proof} %We show the equality for $j=1.$ The case for $j=2$ is proved in the same way. 
As explained in Introduction, the formula (\ref{equ_haidinhnghiatuongduong}) holds for smooth forms $T$ and $S.$ We consider now the general case. We already know that $T \wedge S$ depends continuously on  $S$ for the $*$-topology.  Let $\{S_k\}_{k\in \N}$ be a sequence of smooth forms in $\mathcal{D}_q$ which converges to $S$ in the $*$-topology. Put $\widehat{S}_k= \Pi_2^*(S_k)$ and $R_k=\hat{u} \widehat{T} \wedge \widehat{S}_k.$ It follows from Lemma \ref {co_ukhatichdoivoiTS}  that the masses of $R_k$ converge to the mass of $R=\hat{u} \widehat{T} \wedge \widehat{S}.$ Moreover, $R_k$ converges to $R$ in  $\widehat{X \times X} \backslash \widehat{\Delta}.$
Applying Lemma \ref{le_continuitedeS} to $\widehat{X \times X}$ in the place of $X,$ $R_k$ and $R,$ we see that the right-hand sides of (\ref{equ_haidinhnghiatuongduong}), which is defined in Corollary \ref{co_giaovoiduongcheo}, also depends continuously on  $S$ for the $*$-topology. Hence approximating $S$ by smooth forms allows us to assume that $S$ is smooth. Now Lemma \ref{le_ukhatichvoiS} applied to $\widehat{T}$ in place of $\widehat{S}$  implies that the right-hand side of (\ref{equ_haidinhnghiatuongduong}) is continuous in $T.$ When $S$ is smooth, it is clear that $T \wedge S$ depends continuously on  $T.$ Therefore, (\ref{equ_haidinhnghiatuongduong}) holds since we can approximate $T$ by closed smooth forms. The proof is finished.  
\end{proof} 
%{\bf Acknowledgement.} I would like to thank my doctoral advisor Tien-Cuong Dinh for introducing me in this research direction and his valuable help during the preparation of this paper. The author also wants to express his gratitude to  G. Rivi\`ere for his discussions on the paper \cite{R2}. This research is supported by grants from R\'egion Ile-de-France. %I am also grateful to St\'ephane Nonnenmacher and G. Rivi\`ere for their fruitful discussions on this subject.  %The author is grateful to Tien-Cuong Dinh, Stephane Nonnenmacher and G. Rivieve for fruitful discussions on this subject. 

\bibliography{test2}
\bibliographystyle{siam}

\noindent
Duc-Viet  Vu,
UPMC Univ Paris 06, UMR 7586, Institut de
Math{\'e}matiques de Jussieu-Paris Rive Gauche, 4 place Jussieu, F-75005 Paris, France.\\
{\tt duc-viet.vu@imj-prg.fr}

\end{document}